\documentclass[a4paper,10pt,reqno]{amsart}

\usepackage[utf8]{inputenc}
\usepackage[T1]{fontenc}
\usepackage{lmodern}
\usepackage{amsthm}
\usepackage{amsmath}
\usepackage{amssymb}
\usepackage[inline]{enumitem}
\usepackage{comment}
\PassOptionsToPackage{hyphens}{url}\usepackage{hyperref}
\usepackage{fancyhdr}
\usepackage{mathrsfs}
\usepackage{stmaryrd}
\usepackage[normalem]{ulem}
\usepackage{xcolor}
\usepackage{etoolbox} 
\usepackage{alphalph} 

\setlist[description]{%
  itemsep=0.05cm,               
  font={\normalfont\textsc}, 
 leftmargin=\parindent,
 labelindent=\parindent
}

\theoremstyle{definition}

\newtheorem{theorem}{Theorem}[section]
\newtheorem{lemma}[theorem]{Lemma}
\newtheorem{proposition}[theorem]{Proposition}
\newtheorem{corollary}[theorem]{Corollary}

\theoremstyle{definition}
\newtheorem{definition}[theorem]{Definition}
\newtheorem{example}[theorem]{Example}

\newtheorem{question}[theorem]{Question}
\newtheorem{questions}[theorem]{Questions}

\newtheoremstyle{theoremdd}
  {\topsep}
  {\topsep}
  {\normalfont}
  {0pt}
  {\scshape}
  {:}
  { }
  {\indent\thmname{#1}\thmnumber{ #2}\textnormal{\thmnote{ (#3)}}}

\theoremstyle{theoremdd}
\newtheorem{claimx}{Claim}

\definecolor{blue-url}{RGB}{0,0,100}
\definecolor{red-url}{RGB}{100,0,0}
\definecolor{green-url}{RGB}{0,100,0}
\definecolor{light-yellow}{RGB}{255,255,128}
\definecolor{light-blue}{RGB}{193,255,255}
\definecolor{light-red}{RGB}{239,83,80}

\hypersetup{
	pdftitle={Torsion groups and the Bienvenu-Geroldinger conjecture},
	pdfauthor={},
	pdfmenubar=false,
	pdffitwindow=true,
	pdfstartview=FitH,
	colorlinks=true,
	linkcolor=blue-url,
	citecolor=green-url,
	urlcolor=red-url
}

\renewcommand{\qedsymbol}{$\blacksquare$}
\renewcommand{\emptyset}{\varnothing}
\renewcommand{\setminus}{\smallsetminus}
\renewcommand{\,}{\kern 0.1em}

\providecommand\llb{\llbracket}
\providecommand\rrb{\rrbracket}

\providecommand\ord{{\rm ord}}

\newcommand{\evid}[1]{\textsf{#1}}
\newcommand{\fin}{\mathrm{fin}}

{\newline\vspace{\abovedisplayskip}\hbox to \textwidth\bgroup\hss$\displaystyle}
{$\hss\egroup\vspace{\belowdisplayskip}}

\makeatletter
\DeclareFontFamily{OMX}{MnSymbolE}{}
\DeclareSymbolFont{MnLargeSymbols}{OMX}{MnSymbolE}{m}{n}
\SetSymbolFont{MnLargeSymbols}{bold}{OMX}{MnSymbolE}{b}{n}
\DeclareFontShape{OMX}{MnSymbolE}{m}{n}{
	<-6>  MnSymbolE5
	<6-7>  MnSymbolE6
	<7-8>  MnSymbolE7
	<8-9>  MnSymbolE8
	<9-10> MnSymbolE9
	<10-12> MnSymbolE10
	<12->   MnSymbolE12
}{}
\DeclareFontShape{OMX}{MnSymbolE}{b}{n}{
	<-6>  MnSymbolE-Bold5
	<6-7>  MnSymbolE-Bold6
	<7-8>  MnSymbolE-Bold7
	<8-9>  MnSymbolE-Bold8
	<9-10> MnSymbolE-Bold9
	<10-12> MnSymbolE-Bold10
	<12->   MnSymbolE-Bold12
}{}

\let\llangle\@undefined
\let\rrangle\@undefined
\DeclareMathDelimiter{\llangle}{\mathopen}%
{MnLargeSymbols}{'164}{MnLargeSymbols}{'164}
\DeclareMathDelimiter{\rrangle}{\mathclose}%
{MnLargeSymbols}{'171}{MnLargeSymbols}{'171}
\makeatother

\begin{document}
\title{Torsion groups and the Bienvenu--Geroldinger conjecture}

%
\author{Salvatore Tringali}
\address{(S.~Tringali) School of Mathematical Sciences, Hebei Normal University | Shijiazhuang, Hebei province, 050024 China}
\email{salvo.tringali@gmail.com}
\urladdr{https://salvo-tringali.github.io/home/}

\author{Weihao Yan}
\address{(W.~Yan) School of Mathematical Sciences, Hebei Normal University | Shijiazhuang, Hebei province, 050024 China}
\email{weihao.yan.hebnu@outlook.com}



\begin{abstract}
Equipped with the operation of setwise multiplication induced by a (multiplicatively written) monoid $H$ on its parts, the collection of all finite subsets of $H$ containing the identity element is itself a monoid, denoted by $\mathcal P_{{\rm fin}, 1}(H)$ and called the reduced finitary power monoid of $H$.

One is naturally led to ask whether, for all $H$ and $K$ in a given class of monoids, $\mathcal P_{\fin,1}(H)$ and $\mathcal P_{\fin,1}(K)$  are isomorphic if and only if $H$ and $K$ are. The problem originates from a conjecture of Bienvenu and Geroldinger that was recently settled by the authors. Here, we provide a positive answer to the problem in the case where $H$ and $K$ are cancellative monoids, one of which is torsion. In particular, the answer is in the affirmative when $H$ and $K$ are torsion groups. Whether the conclusion extends to arbitrary groups remains open.
\end{abstract}

\maketitle
\thispagestyle{empty}

\section{Introduction}
\label{sec:intro}

Let $H$ be a semigroup (see the end of this section for notation and terminology). Equipped with the (binary) op\-er\-a\-tion of setwise multiplication defined on the power set of $H$ by
\[
(X,Y) \mapsto \{ xy : x \in X,\, y \in Y \},
\]
the collection of all non-empty subsets of $H$ is itself a semigroup, denoted by 
$\mathcal P(H)$ and called as the \evid{large power semigroup} of $H$. Moreover, the family of all non-empty \textit{finite} subsets of $H$ is a sub\-semi\-group of $\mathcal P(H)$, denoted by $\mathcal P_\fin(H)$ and called the \evid{finitary power semigroup} of $H$. It is clear that $\mathcal P(H) = \mathcal P_\fin(H)$ if and only if $H$ is finite. We will generically refer to either of these structures as a \evid{power semigroup}. 

The systematic investigation of power semigroups began in the late 1960s and was continued quite intensively by semigroup theorists and computer scientists through the 1980s and 1990s. A catalyst of these early developments has been the role that power semigroups play in the study of formal languages and automata \cite{Alm02}.
A milestone in the history of the subject was marked by a 1967 paper of Tamura and Shafer \cite{Tam-Sha1967} that has eventually led to the following questions.

\begin{questions}
\label{ques:tamura-shafer-iso-problem}
Let $\mathscr C$ be a class of semigroups. Is it true that, for all $H, K \in \mathscr C$,
\begin{enumerate}[label=\textup{(\alph{*})}]
\item\label{ques:tamura-shafer-iso-problem(1)} $\mathcal P(H)$ is isomorphic to $\mathcal P(K)$ if and only if $H$ is isomorphic to $K$?
\item\label{ques:tamura-shafer-iso-problem(2)} $\mathcal P_\fin(H)$ is isomorphic to $\mathcal P_\fin(K)$ if and only if $H$ is isomorphic to $K$?
\end{enumerate}
\end{questions}
The interesting aspect of these questions lies in the ``only if\,'' direction. In fact, if $f$ is an isomorphism from a semigroup $H$ to a semigroup $K$, then its \evid{augmentation} 
\begin{equation}\label{equ:augmentation}
f^\ast \colon \mathcal P(H) \to \mathcal P(K) \colon X \mapsto f[X] := \{f(x) \colon x \in X\}
\end{equation}
is a \evid{global isomorphism} from $H$ to $K$, that is, an isomorphism from $\mathcal P(H)$ to $\mathcal P(K)$.
Moreover, $\mathcal P_\fin(H)$ is isomorphic to $\mathcal P_\fin(K)$ via the restriction of $f^\ast$ to the non-empty finite subsets of $H$, since $f^\ast(X) \in \mathcal P_\fin(K)$ for every $X \in \mathcal P_\fin(H)$ (see \cite[Remark 4]{Tri-2024(a)} and \cite[Sect.~1]{GarSan-Tri-24(a)} for additional details).

The answer to Question \ref{ques:tamura-shafer-iso-problem}\ref{ques:tamura-shafer-iso-problem(1)} is negative for the class of all semigroups \cite{Mog1973}; is positive for groups \cite{Shaf-1967}, semilattices \cite[p.~218]{Koba-1984}, Clifford semigroups \cite[Theorem 4.7]{Gan-Zhao-2014}, cancellative commutative semigroups \cite[Corollary 1]{Tri-2024(a)}, etc.;
and is open for finite semigroups \cite[p.~5]{Hami-Nord-2009}, despite some authors having claimed the opposite based on results announced in \cite{Tamu-1987} but never proved.

As for Question \ref{ques:tamura-shafer-iso-problem}\ref{ques:tamura-shafer-iso-problem(2)}, very little is known outside the case where $\mathscr C$ is a class of \textit{finite} semigroups contained in any of the classes that are already covered by the positive results reviewed in the previous paragraph. More precisely, Bienvenu and Geroldinger have established in \cite[Theorem 3.2(3)]{Bie-Ger-25} that the problem has an affirmative answer for numerical monoids, and the conclusion has been subsequently generalized to cancellative commutative semigroups in \cite[Corollary 1]{Tri-2024(a)}. Here, a \evid{numerical monoid} is a sub\-monoid of the additive monoid of non-negative integers with finite complement in $\mathbb N$.

Now suppose that $H$ is a monoid with identity $1_H$. Both $\mathcal P(H)$ and $\mathcal P_\fin(H)$ are then monoids too, their identity being the singleton $\{1_H\}$. Moreover, the family of all finite subsets of $H$ containing $1_H$ is a submonoid of $\mathcal P_\fin(H)$, denoted by $\mathcal P_{\fin,1}(H)$ and called the \evid{reduced finitary power monoid} of $H$. Introduced by Fan and Tringali in \cite{Fa-Tr18} and further investigated in \cite{An-Tr18, Bie-Ger-25, GarSan-Tri-24(a)} and \cite[Sect.~4.2]{Tr20(c)}, $\mathcal P_{\fin,1}(H)$ is one of the most basic objects in an intricate web of structures, generically dubbed as \textsf{power monoids}. 

Power monoids provide an intrinsic algebraic framework for a variety of problems in additive number theory and related fields, including S\'ark\"ozy's conjecture on the ``additive irreducibility'' of the set of [non-zero] quadratic residues of a finite field of prime order \cite[Conjecture 1.6]{Sark2012} and Ostmann's conjecture \cite[p.~13]{Ostm-1968} on the ``asymptotic additive irreducibility'' of the set of (positive rational) primes.
In addition, the arithmetic of power monoids --- particularly with regard to the possibility (or impossibility) of writing certain sets as a finite product of other sets that are, in a suitable sense, irreducible --- has been central to the ongoing development of an ``extended theory of factorization'' \cite{Tr20(c), Tr21(b), Co-Tr-26(a), Co-Tr-21(a), Co-Tr-22(b)} that reaches far beyond the scope of the classical theory \cite{Ger-Hal-06, Ge-Zh-20a}. For a survey of these aspects of the theory, see \cite{Tri-2026}.

With the above ideas in mind, one is naturally led to ask the following question, as a companion to Questions \ref{ques:tamura-shafer-iso-problem}.

\begin{question}\label{ques:BG-like-for-monoids}
Let $\mathscr C$ be a class of monoids. Is it true that, for all $H, K \in \mathscr C$, the monoid $\mathcal P_{\fin,1}(H)$ is isomorphic to $\mathcal P_{\fin,1}(K)$ if and only if $H$ is isomorphic to $K$?
\end{question}

Again, the core of the question lies in proving the ``only if\,'' direction. Indeed, let $f$ be a semigroup iso\-mor\-phism from a monoid $H$ to a monoid $K$, and let $f^\ast$ be the augmentation of $f$, as defined by Eq.~\eqref{equ:augmentation}. Then $f$ maps the identity of $H$ to the identity of $K$ (see, e.g., the last lines of \cite[Sect.~2]{Tri-2024(a)}), and hence $f(u) \in K^\times$ for all $u \in H^\times$. Since $f[X]$ is a finite subset of $K$ for every $X \in \mathcal P_\fin(H)$, it follows that $f^\ast$ restricts to an isomorphism $\mathcal P_{\fin,1}(H) \to \mathcal P_{\fin,1}(K)$ (cf.~\cite[Remark 1.1]{Tri-Yan-23(a)}).

In~\cite[Theorem~2.5]{Tri-Yan-23(a)}, Tringali and Yan answered 
Question~\ref{ques:BG-like-for-monoids} in the affirmative for the class of 
(\evid{rational}) \evid{Puiseux monoids}, that is, submonoids of the non-negative rational numbers 
under addition. This confirmed a conjecture of Bienvenu and 
Geroldinger~\cite[Conjecture~4.7]{Bie-Ger-25} concerning numerical monoids.
In contrast, Rago has recently shown in~\cite[Corollary~3]{Rago2025} that the answer to the same 
question is negative for the class of cancellative commutative monoids --- and, more generally, 
for any class of monoids that contains non-isomorphic cancellative valuation monoids \cite[Theorem and Definition 2.6.4.2]{Halter-Koch2025} with trivial unit groups and isomorphic groups of fractions.

In the present work, we contribute to this line of research by proving that Question \ref{ques:BG-like-for-monoids} admits an affirmative answer for torsion groups (Corollary \ref{cor:4.9}). One major step requires showing --- and this is a fact of independent 
interest --- that, for every isomorphism $f$ from the reduced finitary power monoid of a monoid $H$ to the reduced finitary power 
monoid of a monoid $K$, there exists a bijection 
$g \colon H \to K$, called the pullback of $f$, 
such that $f(\{1_H, x\}) = \{1_K, g(x)\}$ for all $x \in H$ (Corollary 
\ref{cor:iso-induces-a-bijection} and Definition \ref{def:pullback}). Our main result states that if $H$ and $K$ are cancellative monoids and one of them is torsion (and hence a group), then the pullback of an isomorphism from $\mathcal P_{\fin,1}(H)$ to $\mathcal P_{\fin,1}(K)$ is an isomorphism from $H$ to $K$ (Theorem \ref{thm:BG-for-torsion-groups}). Question \ref{ques:BG-like-for-monoids} remains open in the case of arbitrary groups.

\subsection*{Generalities.} If not explicitly specified, we write all semigroups (and monoids) multiplicatively, and all morphisms are \textit{semigroup} homomorphisms (even if their domain and codomain are both monoids). 

We denote by $\mathbb N$ the (set of) non-negative integers, by $\mathbb N^+$ the positive integers, by $\mathbb Z$ the integers, by $|X|$ the cardinality of a set $X$, and by $f^{-1}$ the (functional) inverse of a bijection $f$.
Unless otherwise stated, we reserve the letters $m$ and $n$ for positive integers, and the letters $i$, $j$, and $k$ for non-negative integers.
If $a, b \in \mathbb N$, we let $\llb a, b \rrb := \{x \in \allowbreak \mathbb N \colon \allowbreak a \le x \le b\}$ be the (\evid{discrete}) \evid{interval} from $a$ to $b$. 

An element $a$ in a semigroup $H$ is said to be \evid{torsion} if the set $\{a^n : n \in \mathbb{N}^+\}$ is finite, and 
\evid{cancellative} if left multiplication and right multiplication by $a$ are both injective functions on $H$. The semigroup itself is torsion (respectively, cancellative) if each 
of its elements is.

Suppose now that $H$ is a monoid. The \evid{order} $\ord_H(a)$ of an element $a \in H$ is the \textit{size} of the \evid{cyclic submonoid} 
$\langle a \rangle_H := \{a^n : n \in \mathbb{N}\}$ generated by $a${\kern0.1em}: if $\langle a \rangle_H$ is infinite, then  $\ord_H(a) := \infty$; otherwise, $\ord_H(a)$ is the number of elements in $\langle a \rangle_H$.  It is a basic fact --- and we will often use it later without further comment --- that every cancellative torsion element $a \in H$ is a unit, and its order is the smallest $n \ge 1$ 
such that $a^n = 1_H$. Here, a \evid{unit} is an element $u \in H$ for which there exists a (necessarily unique) element $v \in H$, denoted by $u^{-1}$ and called the \evid{inverse} of $u$, such that $uv = vu = 1_H$.

Additional notation and ter\-mi\-nol\-o\-gy, if not explained when first used, are standard or should be clear from the context. In particular, we refer to Howie's monograph \cite{Ho95} for the basics of semigroup theory.
\section{Preliminaries on the order of an element in a monoid}
\label{sect:02}

Below, we collect some fundamental properties of the order of an element in a monoid $H$, with focus on the case where $H$ is cancellative. These properties are closely related to specific ``structural features'' of the finitary power monoid $\mathcal{P}_{\mathrm{fin},1}(H)$ of $H$ and will be crucial in the proofs of Sect.~\ref{sect:04}.

\begin{lemma}\label{lem_2.1}
Let $H$ be a monoid. A torsion element $z \in H$ has order $n$ if and only if $n$ is the smallest integer $k \ge 1$ such that $\{1_H, z\}^k = \{1_H, z\}^{k-1}$.
\end{lemma}

\begin{proof}
Fix $z \in H$. By \cite[Theorem 1.2.2(2)]{Ho95}, $z$ has order $n$ if and only if $1_H, z, \ldots, z^{n-1}$ are pairwise distinct and there exists $m \in \llb 0, n-1 \rrb$ such that $z^n = z^m$. This implies, on the one hand, that
\begin{equation}\label{lem_2.1:eq_01}
\{1_H, z\}^n = \{1_H, z, \ldots, z^n\} = \{1_H, z, \ldots, z^{n-1}\} = \{1_H, z\}^{n-1};
\end{equation}
and, on the other hand, that
\begin{equation}\label{lem_2.1:eq_02}
z^k \notin \{1_H, z, \ldots, z^{k-1}\} = \{1_H, z\}^{k-1}, \qquad \text{for all } 
k \in \llb 1, n-1 \rrb.
\end{equation}
In particular, Eq.~\eqref{lem_2.1:eq_02} shows that $\{1_H, z\}^k \ne \{1_H, z\}^{k-1}$ for every $k \in \llb 1, n-1 \rrb$, which, together with Eq.~\eqref{lem_2.1:eq_01}, is enough to complete the proof.
\end{proof}

\begin{lemma}\label{lem:smallest-integer}
Let $H$ be a monoid, and fix $z \in H$ and $\ell \in \mathbb N^+$. If $r \ge \ell-1$ is an integer, then 
\begin{equation*}
\{1_H, z^\ell\} \{1_H, z\}^r = \{1_H, z\}^{\ell+r}.
\end{equation*}
If, in addition, $z$ is cancellative and $\ell$ is no larger than the order of $z$ in $H$, then 
\begin{equation*}
\{1_H, z^\ell\} \{1_H, z\}^r \ne \{1_H, z\}^s, \qquad \text{for all } r, s \in \mathbb N \text{ with } r < \ell-1 \le s.
\end{equation*}
\end{lemma}

\begin{proof}
The first part is straightforward: for any integer $r \ge \ell-1$, we have 
\[
\{1_H, z^\ell\} \{1_H, z\}^r = \{z^i: 0 \le i \le r\} \cup \{z^{\ell+i}: 0 \le i \le r\} = \{1_H, z, \ldots, z^{\ell+r}\} = \{1_H, z\}^{\ell+r},
\]
As for the second part, suppose that $z$ is a cancellative element and $\ell$ is no larger than $n := \ord_H(z)$. Assume for a contradiction that there exist
$r, s \in \mathbb N$ with
$r < \ell-1 \le s$ such that
\begin{equation}\label{lem:smallest-integer:eq(02)}
\{1_H, z\}^s = \{1_H, z^\ell\}\{1_H, z\}^r = \{1_H, z\}^r \cup z^\ell \{1_H, z\}^r.
\end{equation}
Since $1 \le \ell \le n$, it is then clear that 
$$
z^{\ell-1} \notin \{1_H, \allowbreak z\}^r
\quad\text{and}\quad
z^{\ell-1} \in \{1_H, \allowbreak z\}^s.
$$
It thus follows 
from Eq.~\eqref{lem:smallest-integer:eq(02)} that $z^{\ell-1} = z^{\ell+i}$ for some 
$i \in \llb 0, r \rrb$. By the cancellativity of $z$, this implies $z^{i+1} = 1_H$. In 
particular, $z$ is a unit and generates a finite cyclic subgroup of the unit group of $H$ of order $n$. Therefore, we conclude from the elementary properties of cyclic groups
\cite[Theorem I.3.4(iv)]{Hung2003} that $n \mid i + \allowbreak 1$, which is impossible since 
$1 \le i+1 \le \allowbreak r + \allowbreak 1 \le \allowbreak \ell-1 < n$.
\end{proof}

\begin{example} 
Let $H$ be a cyclic monoid of order $4$ and index $2$, meaning that $H$ is generated by an element $z$ of order $4$ such that $z^2 = z^4$ (see \cite[Sect.~1.2]{Ho95} for terminology). We have 
$$
\{1_H, z^3\} \{1_H, z\} = \{1_H, z\}^{3+1},
$$
which shows that Lemma \ref{lem:smallest-integer} fails to be true with $\ell = \allowbreak 3$: otherwise, we would have $\{1_H, z^3\} \{1_H, z\}^r \ne \{1_H, z\}^{3+r}$ for each $r \in \{0, 1\}$. In other words, Lemma \ref{lem:smallest-integer} need not be true if $z$ is not cancellative.
\end{example}

\begin{lemma}\label{lem-equation-x^m=y^n}
Let $H$ be a monoid. If $x, y \in H$ and $x^r = y^s$ for some $r, s \in \mathbb N^+$, then
\begin{equation}
\label{equ-x^m=y^n-1} 
\{1_H, x\}^{r-1} \{1_H, xy\} \{1_H, y\}^s = \{1_H, x\}^r \{1_H, y\}^{s+1}
\end{equation}
and
\begin{equation}\label{equ-x^m=y^n-2} 
\{1_H, x\}^r \{1_H, xy\} 
\{1_H, y\}^{s-1} = \{1_H, x\}^{r+1} \{1_H, y\}^s.
\end{equation}
\end{lemma}

\begin{proof}
We focus on the first identity, as the second follows by symmetry.

Denote by $L$ and $R$ the left-hand side and the right-hand side of Eq.~\eqref{equ-x^m=y^n-1}, respectively. Given $z \in H$, it is clear that $z \in L$ if and only if either $z = x^i y^j$ or $z = x^{i+1}y^{j+1}$ for some $i \in \llbracket 0, r-1 \rrbracket$ and $j \in \llbracket 0, s \rrbracket$, while $z \in R$ if and only if $z = x^i y^j$ for some $i \in \llbracket 0, r \rrbracket$ and $j \in \llbracket 0, s+1 \rrbracket$. Consequently, we have
$$
R \setminus \{x^r, y^{s+1}\} \subseteq L \subseteq R.
$$
It remains to check that $x^r, y^{s+1} \in L$, and this is immediate. Indeed, we have $x^r = \allowbreak y^s$ (by hypothesis), and hence $y^{s+1} = x^r y$. Since $y^s, x^r y \in L$ (as noted above), we are done.
\end{proof}

Lemma \ref{lem-equation-x^m=y^n} applies, in particular, to torsion 
groups, because in a torsion group every element has a positive power equal to 
the identity. We conclude the section with the following:

\begin{proposition}\label{prop:x^r=y^s}
Let $H$ be a monoid, and let $x, y \in H$ be cancellative torsion elements. Then $x$ and $y$ are units, and there exist $r, u \in \llb 1, \ord_H(x) \rrb$ and $s, v \in \llb 1, \ord_H(y) \rrb$ such that
\begin{equation*}
\text{(i) } x^r = y^s \text{ and } x^u = y^v;
\qquad
\text{(ii) if } x^c = y^d \text{ for some } c, d \in \mathbb Z, \text{ then }
r \mid c \text{ and } v \mid d. 
\end{equation*}
\end{proposition}

\begin{proof} 
Set $m := \ord_H(x)$ and $n := \ord_H(y)$. Every cancellative torsion element $z \in H$ is a unit, and its order is the smallest integer $k \ge 1$ such that $z^k = 1_H$. Accordingly, we have $x^m = y^n = 1_H$, which, by the well-ordering principle, implies the existence of integers $r, u \in \llb 1, m \rrb$ and $s, v \in \llb 1, n \rrb$ such that 
\begin{equation}\label{pro-r,s:equ_01}
\text{(I) } x^r = y^s \text{ and } x^u = y^v;
\qquad
\text{(II) if } c, d \in \mathbb N^+ \text{ and } x^c = y^d, \text{ then }
r \le c \text{ and } v \le d. 
\end{equation}
With that said, we focus on proving part~(ii) of the statement for $r$, as the case for $v$ is symmetric.

First, assume $x^c = y^d$ with $c \in \mathbb N^+$ and $d \in \mathbb Z$. By the Division Algorithm \cite[p.~11]{Hung2003}, there exist $q \in \mathbb Z$ and $k \in \allowbreak \llb 0, \allowbreak n-1 \rrb$ such that $d = nq + k$. If $k = 0$, then $x^c = \allowbreak (y^n)^q = 1_H$, and hence, by the elementary properties of cyclic groups
\cite[Theorem I.3.4(iv)]{Hung2003}, $m \mid c$. Otherwise, $x^c = \allowbreak y^{nq + k} = \allowbreak (y^n)^q y^k = y^k$, which, by Eq.~\eqref{pro-r,s:equ_01} and the positivity of $k$, yields $r \le \allowbreak c$. In any case (recall that $r \le m$), this shows that
\begin{equation}\label{pro-r,s:equ_02}
\text{if } x^c = y^d \text{ for some }
c \in \mathbb N^+ \text{ and }
d \in \mathbb Z,
\text{ then }
r \le c.
\end{equation}

Next, suppose that $x^c = y^d$ with $c, d \in \mathbb Z$. Similarly as before, there exist $q \in \mathbb Z$ and $k \in \llb 0, r-1 \rrb$ such that $c = rq + k$, with the result that $y^d = x^{rq} x^{k} = y^{sq} x^k$ and hence $x^k = y^{d-sq}$. By Eq.~\eqref{pro-r,s:equ_02}, this is only possible if $k = 0$, which implies that $r \mid c$ and completes the proof.
\end{proof}

\section{The two-to-two property}
\label{sect:2-to-2-property}

Let $H$ and $K$ be monoids. In this short section, we show that all isomorphisms from the reduced finitary power monoid $\mathcal{P}_{\fin,1}(H)$ of $H$ to the reduced finitary power monoid of $K$ have a somewhat surprising property in common: they bijectively map sets of the form $\{1_H, x\} \subseteq H$ to sets of the form $\{1_K, \allowbreak y\} \subseteq \allowbreak K$, regardless of any hypotheses on $H$ and $K$ (Corollary~\ref{cor:iso-induces-a-bijection}). 
The proof is essentially combinatorial, as the main step (Theorem~\ref{thm:2-element-sets-to-2-elements-sets}) reduces to counting and bounding the number of solutions to certain equations involving product sets.

\begin{lemma}
\label{lemma:nr-of-reps}
Let $H$ be a monoid, $S$ a subset of $H$ containing the identity $1_H$, and $n$ an integer $\ge 3$. If $T \subseteq S$ and $1_H \notin T$, then $(S^{n-1} \setminus T)S = S^n$. In particular, the equation $AS = S^n$ has at least $2^{|S|-1}$ solutions $A \in \mathcal P(H)$ such that $1_H \in A$, and each of these solutions is finite whenever $S$ is.
\end{lemma}

\begin{proof}
Let $T$ be a subset of $S \setminus \{1_H\}$, and define $Q := S^{n-1} \setminus T$. It is obvious that $QS \subseteq \allowbreak S^{n-1} S = \allowbreak S^n$, so we only need to check that $S^n \subseteq QS$. 
To this end, fix $z \in S^n$, and let $k$ be the smallest non-negative integer such that $z \in S^k$ (it is clear that $0 \le k \le n$; and since $1_H \in S$, we have $S \subseteq S^2 \subseteq \cdots \subseteq S^n$). Our goal is to show that $z \in QS$, and we distinguish three cases.

\begin{itemize}[leftmargin=0.3in]
\item\textsc{Case 1:} $k = 0$ or $k = 1$. If $k = 0$, then $z \in S^0 = \{1_H\}$ and thus $z = 1_H$. It follows that, regardless of whether $k = 0$ or $k = 1$, $z \in S$ and hence $z \in 1_H S \subseteq QS$ (note that $1_H \in Q$).

\vskip 0.05cm

\item\textsc{Case 2:} $2 \le k \le n-1$. We have $z \notin T$, or else $z \in S$ (by the fact that $T \subseteq S$), contradicting the definition itself of $k$ (which guarantees that $z \notin S^i$ for every non-negative integer $i < k$). It follows that $z \in S^k \setminus T \subseteq Q$ (by the fact that $S^k \subseteq S^{n-1}$), and hence $z \in Q1_H \subseteq QS$.

\vskip 0.05cm

\item\textsc{Case 3:} $k = n$. By the definition of $k$, we have that $z \notin S^i$ for each $i \in \llb 1, n-1 \rrb$. On the other hand, $z = s_1 \cdots s_n$ for some $s_1, \ldots, s_n \in S$. It follows that $z' := s_1 \cdots s_{n-1} \notin T$, or else $z = z' s_n \in S^2$, which is a contradiction because $2 < 3 \le n$. Therefore, $z' \in S^{n-1} \setminus T$ and hence $z = z's_n \in Q S$.
\end{itemize}

As for the second part of the statement, we have already observed that $S \subseteq S^{n-1}$. So, if $T_1$ and $T_2$ are distinct subsets of $S \setminus \{1_H\}$, then the sets $S^{n-1} \setminus T_1$ and $S^{n-1} \setminus T_2$ are likewise distinct. As a result, the equation $AS = S^n$ has at least as many solutions $A \in \mathcal P(H)$ with $1_H \in A$ as there are subsets of $S \setminus \{1_H\}$; namely, at least $2^{|S|-1}$ solutions. Moreover, if $S$ is finite, then each of these solutions is finite as well.
\end{proof}

\begin{theorem}
\label{thm:2-element-sets-to-2-elements-sets}
Let $f$ be an isomorphism from $\mathcal P_{\fin,1}(H)$ to $\mathcal P_{\fin,1}(K)$, where $H$ and $K$ are arbitrary monoids. Then $f(X)$ is a $2$-element set for every $2$-element set $X \in \mathcal P_{\fin,1}(H)$.
\end{theorem}

\begin{proof}
The claim is vacuously true if $|H| = 1$. Otherwise, fix a non-identity element $x \in H$, and set 
\[
X := \allowbreak \{1_H, x\},
\quad
Y := f(X),
\quad\text{and}
\quad
k := |Y| - 1. 
\]
Considering that $1_K \in Y$ and $Y \ne \allowbreak f(\{1_H\}) = \allowbreak \{1_K\}$ (by the in\-jec\-tivity of $f$), it is clear that $k$ is a \emph{positive} integer. We need to prove that $k = 1$. To this end, let us define
$$
\mathcal{C}_X := \{A \in \mathcal P_{\fin,1}(H) \colon AX = X^3\}
\quad\text{and}\quad
\mathcal{D}_Y := \{B \in \mathcal P_{\fin,1}(K) \colon BY = Y^3\}. 
$$
Since $1_H \in X$, it is clear that $A \in \mathcal{C}_X$ yields $A \subseteq X^3$ and $f(A) Y = Y^3$. In a similar way, $B \in \mathcal{D}_Y$ yields $B \subseteq Y^3$ and $f^{-1}(B) X = X^3$. We thus see that $d := \allowbreak |\mathcal{C}_X| = \allowbreak |\mathcal{D}_Y| \in \mathbb N^+$, for $X^3$ is a finite set and the restriction of $f$ to $\mathcal{C}_X$ establishes a bijection $\mathcal{C}_X \to \mathcal{D}_Y$.

With these preliminaries in place, suppose for a contradiction that $k \ge 2$. We then get from the above and Lemma \ref{lemma:nr-of-reps} (applied with $S = Y$, $n = 3$, and $K$ in place of $H$) that $d = |\mathcal{D}_Y| \ge 4$. Let $m$ be the order of the sub\-monoid $\{1_H, x, x^2, \ldots\}$ of $H$ generated by $x$; note that $m \ge 2$ (by the fact that $x \ne 1_H$). We will distinguish four cases, depending on whether $m = 2$, $m = 3$, $m = 4$, or $m \ge 5$.

\vskip 0.05cm

\textsc{Case 1:} $m = 2$. We have $x^2 \in X$ and hence $X^3 = X$. Consequently, $A \in \mathcal{C}_X$ only if $A \subseteq X$, which shows in turn that $d = |\mathcal{C}_X| = 2$ and contradicts our previous finding (that $d \ge 4$).

\vskip 0.05cm

\textsc{Case 2:} $m = 3$. 
We have $x^3 \in X^2$ and $x^2 \notin X$, with the result that $X^3 = X^2 \ne X$ and, hence, $A \in C_X$ only if $\{1_H\} \subsetneq A \subseteq X^3 = \{1_H, x, x^2\}$. It follows that $d = |\mathcal{C}_X| = 3$, which is again absurd. 

\vskip 0.05cm

\textsc{Case 3:} $m = 4$. We have $x^4 \in X^3$ and $x^3 \notin X^2$ (or else $m \le 3$), implying that $X^2 \ne X^3 = X^4$ and hence $Y^2 \ne Y^3 = Y^4$. Since $Y \subseteq Y^2 \subseteq Y^3$, it follows from Lemma~\ref{lemma:nr-of-reps} (applied with $K$ in place of $H$, first for $n = 3$ and then for $n = 4$) that, for \textit{all} subsets $T$ and $T{\kern0.1em}'$ of $Y$ not containing the identity $1_H$, the sets
$Y^2 \setminus T$ and $Y^3 \setminus T{\kern0.1em}'$ are \textit{distinct} solutions $B \in \mathcal P_{\fin,1}(K)$ to the equation $BY = Y^3$ (note that, since $Y^3 = Y^4$, every solution $B$ to $BY = Y^3$ is also a solution to $BY = Y^4$, and vice versa). Consequently, 
\[
d = |\mathcal{D}_Y| \ge 2^k + 2^k \ge 8. 
\]
If, on the other hand, $A \in \allowbreak \mathcal{C}_X$, then $\{1_H\} \neq \allowbreak A \ne \allowbreak \{1_H, x\}$ (or else $AX \subseteq X^2 \subsetneq X^3$). Namely, $A \in \mathcal{C}_X$ only if $A = \allowbreak \{1_H\} \cup \allowbreak A'$ for some $A' \subseteq \allowbreak X^3 \setminus \{1_H\} = \allowbreak \{x, x^2, x^3\}$ with $\emptyset \ne A' \ne \{x\}$. But this means that $8 \le d = \allowbreak |\mathcal{C}_X| \le 2^3 - 1 = 7$, which is still a contradiction. 

\vskip 0.05cm

\textsc{Case 4:} $m \ge 5$. We have $x^4 \notin X^3$ (or else $m \le 4$) and $x^3 \notin X^2$ (or else $m \le 3$). So, if $A \in \mathcal{C}_X$, then $x^3 \notin A$ and hence $A \subseteq X^2$; moreover, $A \ne \{1_H\}$ (or else $AX = X \ne X^3)$ and $x^2 \in A$ (or else $A \subseteq X$ and hence $AX \subseteq X^2 \subsetneq X^3$). It follows that $A \in \mathcal{C}_X$ if and only if $A = \{1_H, x^2\}$ or $A = X^2$, with the result that $d = |\mathcal{C}_X| = 2$. However, this is once more impossible and finishes the proof.
\end{proof}

\begin{corollary}
\label{cor:iso-induces-a-bijection}
Let $f$ be an isomorphism from $\mathcal P_{\fin,1}(H)$ to $ \mathcal P_{\fin,1}(K)$, where $H$ and $K$ are arbitrary monoids. There is then a \textup{(}uniquely determined\textup{)} bijection $g \colon H \to K$ such that $f(\{1_H, \allowbreak x\}) = \allowbreak \{1_K, g(x)\}$ for each $x \in H$. In particular, $g(1_H) = 1_K$.
\end{corollary}

\begin{proof}
Since $f$ maps $\{1_H\}$ to $\{1_K\}$, we gather from Theorem \ref{thm:2-element-sets-to-2-elements-sets} that, for all $x \in H$, there is a (unique) element $y \in K$ with the property that $f(\{1_H, x\}) = \{1_K, y\}$. Namely, there exists a function $g \colon H \to K$ such that $g(1_H) = 1_K$ and $f(\{1_H, x\}) = \{1_K, g(x)\}$ for all $x \in H$. It remains to see that $g$ is bijective.

To begin, fix $u, v \in H$ with $u \ne v$. We have from the above and the injectivity of $f$ that 
$$
\{1_K, g(u)\} = \allowbreak f(\{1_H, u\}) \ne f(\{1_H, v\}) = \{1_K, g(v)\}, 
$$
which is only possible if $g(u) \ne g(v)$. Consequently, $g$ is itself injective.

Next, let $y \in K$. Since $f$ is an isomorphism $\mathcal P_{\fin,1}(H) \to \mathcal P_{\fin,1}(K)$, its inverse $f^{-1}$ is an isomorphism from $\mathcal P_{\fin,1}(K)$ to $\mathcal P_{\fin,1}(H)$. So, again from Theorem \ref{thm:2-element-sets-to-2-elements-sets}, we have that $f^{-1}(\{1_K, y\}) = \allowbreak \{1_H, x\}$ for some $x \in H$. It follows that $\{1_K, y\} = f(\{1_H, x\}) = \{1_K, g(x)\}$ and hence $g(x) = y$, which shows that $g$ is also surjective and completes the proof.
\end{proof}

It remains an open question whether Theorem~\ref{thm:2-element-sets-to-2-elements-sets} can be generalized to show that, for arbitrary monoids $H$ and $K$, every isomorphism from $\mathcal P_{\fin,1}(H)$ to $\mathcal P_{\fin,1}(K)$ is cardinality-preserving (the theorem yields a positive answer for sets of cardinality one or two); this remains unclear even in the cancellative case. Nevertheless, the result provides a far-reaching generalization of \cite[Lemma 2.3]{Tri-Yan2023(b)} and motivates the following definition, which plays a crucial role in this work.

\begin{definition}\label{def:pullback}
Let $H$ and $K$ be monoids. Based on Corollary \ref{cor:iso-induces-a-bijection}, we let the \evid{pullback} of an isomorphism $f \colon \mathcal P_{\fin,1}(H) \to \mathcal P_{\fin,1}(K)$ be the unique bijection $H \to K$ such that $f(\{1_H, x\}) = \{1_K, g(x)\}$ for all $x \in H$.
\end{definition}

We devote the next section to studying some of the properties of the pullback. The ultimate goal is to prove that, under certain circumstances, the pullback is actually an isomorphism (Proposition \ref{pro-main}).
\section{The pullback and its properties}
\label{sect:04}

Most of the results proved in this section (specifically, Propositions \ref{prop:g(x^k)=g(x)^k} and \ref{prop:g(x^k)=g(x)^k}, as well as 
Lemmas \ref{lem:g(xy)=g(x)g(y)orx^2y^2=1} and \ref{lem-x^2=1}) rely on cancellativity, and it would be interesting to see whether they extend to more general settings. 
More precisely, we first show that pullbacks (Definition~\ref{def:pullback}) 
are order-preserving, which holds for arbitrary monoids. We then establish 
that, at least for cancellative monoids, they map $k$-powers to $k$-powers 
for all $k \in \mathbb{N}$ (Proposition~\ref{prop:g(x^k)=g(x)^k}).

\begin{proposition}\label{prop:fix-ord}
Let $g$ be the pullback of an isomorphism $f$ from $\mathcal P_{\fin,1}(H)$ to $\mathcal P_{\fin,1}(K)$, where $H$ and $K$ are arbitrary monoids. Then $\mathrm{ord}_H(x) = \mathrm{ord}_K(g(x))$ for every $x \in H$.
\end{proposition}

\begin{proof}
Fix an element $x \in H$, and denote by $n$ its order. Our objective is to show that $m := \mathrm{ord}_K(g(x)) = \allowbreak n$. We distinguish two cases, depending on whether $n$ is finite or not.

\vskip 0.15cm

\textsc{Case 1:} $n < \infty$. By Lemma~\ref{lem_2.1}, $n$ is the smallest integer $k \geq 1$ such that 
\[
    \{1_H, x\}^k = \{1_H, x\}^{k-1}.
\]
Since, for all $X \in \mathcal{P}_{\mathrm{fin},1}(H)$ and $r \in \mathbb{N}$, the equality $X^r = X^{r-1}$ holds if and only if $f(X)^r = f(X)^{r-1}$, it follows that $n$ is also the smallest integer $k \geq 1$ for which 
\[
    \{1_K, g(x)\}^k = f(\{1_H, x\})^k = f(\{1_H, x\})^{k-1} = \{1_K, g(x)\}^{k-1}.
\]
Therefore, by reapplying Lemma~\ref{lem_2.1}, we conclude that $m = n$.

\vskip 0.15cm

\textsc{Case 2:} $n = \infty$. Given  $i, j \in \mathbb{N}$ with $i < j$, we infer from the basic properties of cyclic semigroups \cite[Theorem~1.2.2(1)]{Ho95} that 
$x^i \neq x^j$. This in turn implies
$\{1_H, x\}^i \subsetneq \{1_H, x\}^j$, and consequently
\[
    \{1_K, g(x)\}^i = f\bigl(\{1_H, x\}^i\bigr) \neq f\bigl(\{1_H, x\}^j\bigr) = \{1_K, g(x)\}^j.
\]
It follows that $m = \infty = n$, completing the proof.
\end{proof}

\begin{lemma}\label{lem:x^k-2k<ord}
Let $g$ be the pullback of an isomorphism $f \colon \mathcal P_{\fin,1}(H) \to \mathcal P_{\fin,1}(K)$, where $H$ and $K$ are monoids with $K$ cancellative. If $x \in H$ and $k \in \mathbb N$, then $g(x^k) = g(x)^\ell$ for some $\ell \in \mathbb N$ with $\ell \le k$.
\end{lemma}

\begin{proof}
Fix $x \in H$ and $k \in \mathbb{N}$, and set $y := g(x)$. We need to show that $g(x^k) = y^\ell$ for some $\ell \in \llbracket 0, k \rrbracket$. 

To start with, we have from Corollary \ref{cor:iso-induces-a-bijection} and Proposition \ref{prop:fix-ord} that
\begin{equation}
\label{lem:x^k-2k<ord:equ_01}
g(x^0) = g(1_H) = 1_K = y^0
\quad\text{and}\quad
n := \ord_H(x) = \ord_K(y).
\end{equation}
Since $x^k = x^i$ for some $i \in \llb 0, n-1 \rrb$ (including the case $n = \infty$, where $\llb 0, n-1 \rrb = \mathbb N$ and $i = k$), 
we will therefore assume in the remainder that $1 \leq k < n$. By the first part of Lemma~\ref{lem:smallest-integer}, this yields 
$$
\{1_H, x^k\} \{1_H, x\}^{k-1} = \{1_H, x\}^{2k-1},
$$
which, by taking the image of both sides under $f$, leads to
\begin{equation*}
g(x^k) \in \{1_K, g(x^k)\} \{1_K, y\}^{k-1} =  
\{1_K, y\}^{2k-1}.
\end{equation*}
It follows that $g(x^k) = y^\ell$ for some $\ell \in \llb 0, 2k-1 \rrb$, and consequently
\begin{equation}
\label{lem:x^k-2k<ord:equ_02}
\{1_K, y^\ell\} \{1_K, y\}^{k-1} = \{1_K, y\}^{2k-1}.
\end{equation}
Considering that $g$ is injective and $1 \leq k < n$, we may actually assume by Eq.~\eqref{lem:x^k-2k<ord:equ_01} that 
$$
1 \le \ell \le \min(n-1, 2k-1);
$$
in particular, $n$ being the order of $y$ in $K$ guarantees that $y^\ell = y^j$ for some $j \in \llb 0, n-1 \rrb$. Since $K$ is, by hypothesis, a cancellative monoid, it is then clear from Eq.~\eqref{lem:x^k-2k<ord:equ_02} and the second part of Lemma~\ref{lem:smallest-integer} (applied with $z = y$, $r = \allowbreak k - \allowbreak 1$, $s = \allowbreak 2k-\allowbreak 1$, and $K$ in place of $H$) that $\ell - 1 \leq \allowbreak k - \allowbreak 1$, namely, $\ell \leq k$.
\end{proof}

\begin{proposition}\label{prop:g(x^k)=g(x)^k}
Let $g$ be the pullback of an isomorphism $f$ from $\mathcal P_{\fin,1}(H)$ to $\mathcal P_{\fin,1}(K)$, where $H$ and $K$ are cancellative monoids. Then $g(x^k) = g(x)^k$ for all $x \in H$ and $k \in \mathbb N$.
\end{proposition}

\begin{proof}
Fix $x \in H$ and $k \in \mathbb{N}$. The inverse $f^{-1}$ of $f$ is an isomorphism $\mathcal{P}_{\fin,1}(K) \to \mathcal{P}_{\fin,1}(H)$. Accordingly, let $h$ denote the pullback of $f^{-1}$, and set $y := g(x)$ and $z := h(y)$.

First, observe that $x^k = x^r$ for some $r \in \llbracket 0, n-1 \rrbracket$, where  
$n := \ord_H(x)$. Since $H$ is cancellative, we additionally have $k \equiv r \bmod n$,  
with the convention that this congruence reduces to $k = r$ when $n = \infty$.
On the other hand, Lemma~\ref{lem:x^k-2k<ord} guarantees that there exist $\ell, m \in \mathbb{N}$ with $m \le \ell \le r$ such that $g(x^r) = y^\ell$  
and $h(y^\ell) = z^m$; and it is actually clear from our definitions that $h = g^{-1}$ and hence $z = x$, where $g^{-1}$ denotes the inverse of $g$ (recall from Corollary~\ref{cor:iso-induces-a-bijection} that $g$ is a bijection).

Thus, $x^r = g^{-1}(y^\ell) = x^m$. Since $0 \le m \le r < n = \ord_H(x)$, it follows that $r = m$. Consequently, we gather from the above that $\ell = r$  
and $g(x^k) = g(x^r) = y^r$. Now, considering that $K$ is also a cancellative monoid and, by Proposition~\ref{prop:fix-ord}, the order of $y$ in $K$ is the same as the order of $x$ in $H$, we conclude from the congruence $k \equiv r \bmod n$ that $g(x^k) = y^r = y^k = g(x)^k$.
\end{proof}

\begin{example}
Let $H$ be a cyclic group of order $2$, and let $K$ be an idempotent monoid of order $2$. 
Note that $H$ and $K$ are unique up to isomorphism; in particular, $K$ can be realized as the submonoid $\{0,1\}$ of the non-negative integers under multiplication.

The function $f \colon \mathcal P_{\fin,1}(H) \to \mathcal P_{\fin,1}(K)$ that maps $\{1_H\}$ to $\{1_K\}$ and $H$ to $K$ is a (monoid) isomorphism, whose pullback is the map $g \colon H \to K$ sending $1_H$ to $1_K$ and the unique non-identity element $x \in H$ to the unique non-identity element $y \in K$. 
However, we have $x^2 = 1_H$, and hence
\[
1_K = g(1_H) = g(x^2) \ne g(x)^2 = y.
\]
This ultimately proves that Proposition \ref{prop:g(x^k)=g(x)^k} need not be true when the monoids in its statement are not \textit{both} cancellative.
\end{example}

We continue with a series of lemmas that will eventually allow us to prove that, in the  
cancellative setting, pullbacks behave like semigroup homomorphisms when restricted  
to the product of two torsion elements (Proposition \ref{pro-main}).

\begin{lemma}
\label{lem:g(xy)=g(x)g(y)orx^2y^2=1}
Let $g$ be the pullback of an isomorphism $f$ from $\mathcal{P}_{\fin,1}(H)$ to $\mathcal{P}_{\fin,1}(K)$, where $H$ and $K$ are cancellative monoids. If $x, y \in H$ are torsion elements, then either $g(xy) = g(x)g(y)$ or $x^2 y^2 = 1_H$.
\end{lemma}

\begin{proof}
Set $m := \ord_H(x)$ and $n := \ord_H(y)$, and let $a := g(x)$ and $b := g(y)$. We will often use without further comment that $g$ is an injective function and each of $H$ and $K$ is a cancellative monoid.

By the hypothesis that $x$ and $y$ are torsion elements, $m$ and $n$ are positive integers. Furthermore, we know from Prop\-o\-si\-tion \ref{prop:x^r=y^s} that $x, y \in H^\times$, and there exist $r, \allowbreak u \in \allowbreak \llb 1, m \rrb$ and $s, v \in  \llb 1, n \rrb$ such that 
\begin{equation}\label{equ-r|a,v|b}
\text{(i) }
x^r = y^s \text{ and } x^u = y^v;
\qquad
\text{(ii) if } 
x^c = y^d 
\text{ for some }
c, d \in \mathbb Z, \text{ then } r \mid c \text{ and }v \mid d.
\end{equation}
We henceforth assume that $g(xy) \ne g(x) g(y)$, that is, $g(xy) \ne ab$; otherwise, there is nothing to do.
To facilitate the presentation, the remainder of the proof is structured as a sequence of claims.

\medskip

\begin{claimx}
\label{lemma_4.5:claim_A}
$y \ne x^c$ and $x \ne y^d$ for all $c, d \in \mathbb Z$. 
\end{claimx}

\begin{proof}[Proof of \textsc{Claim} \ref{lemma_4.5:claim_A}]
Suppose, for the sake of contradiction, that $y = x^c$ for some $c \in \mathbb{Z}$ (the other case is similar).  
Using that $x^m = 1_H$ and applying the Division Al\-go\-rithm to write $c = \allowbreak nq + k$ with  
$q \in \allowbreak \mathbb{Z}$ and $k \in \allowbreak \llb 0, m - 1 \rrb$, we get that $
x^c = \allowbreak (x^m)^q x^k = \allowbreak x^k$. 
It is then clear by
Proposition~\ref{prop:g(x^k)=g(x)^k} that
\begin{equation*}
ab \ne g(xy) = g(x^{c+1}) = g(x^{k+1}) = a^{k+1} =  a \cdot a^k = a g(x^k) = a  g(x^c) = a g(y) = ab,
\end{equation*}
which is impossible and implies that $y \ne x^c$ for all $c \in \mathbb Z$. 
\renewcommand{\qedsymbol}{[\textit{Proof of \textsc{Claim} \ref{lemma_4.5:claim_A}}]\,$\square$} 
\end{proof}

\begin{claimx}
\label{lemma_4.5:claim_B}
$r \ge 2$ and $v \ge 2$.
\end{claimx}

\begin{proof}[Proof of \textsc{Claim} \ref{lemma_4.5:claim_B}]
Upon considering that $r$ and $v$ are positive integers, the conclusion is straightforward from Eq.~\eqref{equ-r|a,v|b}(i) and \textsc{Claim} \ref{lemma_4.5:claim_A}. \renewcommand{\qedsymbol}{[\textit{Proof of \textsc{Claim} \ref{lemma_4.5:claim_B}}\,]\,$\square$} 
\end{proof}

\begin{claimx}\label{lemma_4.5:claim_C}
    There exist $j, \ell \in \mathbb N$ with $2 \le j < n$ and $2 \le \ell < m$ such that $g(xy) =  ab^j = a^\ell b$.
\end{claimx}

\begin{proof}[Proof of \textsc{Claim} \ref{lemma_4.5:claim_C}]
We focus on proving that $g(xy) = ab^j$ for some $j \in \llb 2, n-1 \rrb$, as the second half of the statement can be established from Eq.~\eqref{equ-r|a,v|b} and Lemma \ref{lem-equation-x^m=y^n} by a symmetric argument. 

To start with, we gather from Eq.~\eqref{equ-r|a,v|b}(i) and the first part of Lemma~\ref{lem-equation-x^m=y^n} that 
\begin{equation}
\label{prop:special-form-of-g(xy):eq-1}
\{1_H, x\}^{r-1} \{1_H, xy\} \{1_H, y\}^s = \{1_H, x\}^r \{1_H, y\}^{s+1}.
\end{equation}
By applying $f$ to Eq.~\eqref{prop:special-form-of-g(xy):eq-1}, it follows that
\begin{equation}\label{equ-a^r-b^s(1)}
g(xy) \in \{1_K, a\}^{r-1} \{1_K, g(xy)\} \{1_K, b\}^s = \{1_K, a\}^r \{1_K, b\}^{s+1}.
\end{equation}
Therefore, $g(xy) = a^i b^j$ for some $i, j \in \mathbb N$ with $i \le r$ and $j \le s+1$. Let us check that $i = 1$.

Proposition \ref{prop:fix-ord} ensures that $\mathrm{ord}_K(b) = n$. Since $K$ is a cancellative monoid, $b$ is then a unit of order $n$ in $K$, and hence $b^z \{1_K, b\}^n = \{1_K, b\}^n$ for all $z \in \mathbb Z$. As a result, 
\begin{equation*}
\label{prop:special-form-of-g(xy):eq-2}
\{1_K, a^i\} \{1_K, b\}^n = \{1_K, b\}^n \cup a^i \{1_K, b\}^n =  \{1_K, b\}^n \cup a^i b^j \{1_K, b\}^n = \{1_K, a^i b^j\} \{1_K, b\}^n,
\end{equation*}
which, by the injectivity of $f$ (and omitting some routine steps), leads to
\begin{equation*}
x^i \in \{1_H, x^i\} \{1_H, y\}^n = \{1_H, xy\} \{1_H, y\}^n.
\end{equation*}
Thus, either $x^i = y^k$ or $x^i = x y^{k+1}$ for some $k \in \llb 0, n \rrb$.  
However, the first possibility must be ruled out. Indeed, we proved above that $i$ is a \textit{positive} integer no larger than $r$.  
So, if $x^i = y^k$, then by Eq.~\eqref{equ-r|a,v|b}(ii), we can only have $i = r$.  
This, however, entails by Proposition~\ref{prop:g(x^k)=g(x)^k} that
$$
g(xy) = a^r b^j = g(x^r)\, b^j = g(y^s)\, b^j = b^s b^j = b^{s+j} = g(y^{s+j}),
$$
which yields $xy = y^{s+j}$ and thus 
$x = y^{s + j - 1}$, contradicting \textsc{Claim}~\ref{lemma_4.5:claim_A}.

To sum up, we have found that $x^i = x y^{k+1}$, and hence 
$x^{i - 1} = y^{k + 1}$. By Eq.~\eqref{equ-r|a,v|b}(ii), this proves that $r \mid i - 1$. Since $r \ge 2$ (\textsc{Claim} \ref{lemma_4.5:claim_B}) and $-1 \le \allowbreak i - 1 < r$, it is then clear that $i = 1$ and, consequently, $g(xy) = ab^j$. 
It remains to see that, without loss of generality, $2 \le j < n$.  

For, we may of course assume that $1 \le \allowbreak j \le n$, as we have already noted that $b$ is a unit of order $n$ in $K$.  Moreover, the condition $ab^j = g(xy) \ne ab$ ensures that $j \ne 1$, and thus $2 \le j \le n$. Finally, suppose for a contradiction that $j = n$. Then 
$
g(xy) = \allowbreak ab^n = \allowbreak a = g(x)$, 
and hence $y = 1_H$.  
Since $g(1_H) = 1_K$, it follows that $ab \ne \allowbreak g(xy) = \allowbreak g(x) = \allowbreak a g(1_H) = ab$, which is absurd.
\renewcommand{\qedsymbol}{[\textit{Proof of \textsc{Claim} \ref{lemma_4.5:claim_C}}\,]\,$\square$} 
\end{proof}

\begin{claimx}\label{lemma_4.5:claim_D}
$s \le n-2$ and $u \le m-2$.
\end{claimx}

\begin{proof}[Proof of \textsc{Claim} \ref{lemma_4.5:claim_D}]
We verify that $s \le n-2$, as the other inequality can be treated analogously. 

We already know that $1 \le s \le n$. By \text{Claim} \ref{lemma_4.5:claim_C} (and the cancellativity of $K$), there exist $j, \ell \in \mathbb N$ with $2 \le j < \allowbreak n$ and $2 \le \ell < m$ such that $a^{\ell-1} = b^{j-1}$. By Proposition \ref{prop:g(x^k)=g(x)^k}, this guarantees that 
$$
g(x^{\ell - 1}) = a^{\ell-1} = b^{j-1} = g(y^{j-1}).
$$
Consequently, $x^{\ell - 1} = y^{j - 1}$, and so, by Eq.~\eqref{equ-r|a,v|b}(i), we have  
$r \mid \ell - 1$. It follows that $r \le \ell - 1 \le m - 2$, and hence  
$s \le n - 1$; in fact, if $s = n$, then $x^r = y^s = 1_H$, and thus $r = m$,  
a contradiction.

If, on the other hand, $s = n-1$, then by Eq.~\eqref{equ-r|a,v|b}(i) we have  
$x^r = y^{n - 1} = y^{-1}$, and hence $y = x^{-r}$, which is impossible by 
\textsc{Claim}~\ref{lemma_4.5:claim_A}. Therefore,  
$s \le n - 2$.
\renewcommand{\qedsymbol}{[\textit{Proof of \textsc{Claim} \ref{lemma_4.5:claim_D}}\,]\,$\square$} 
\end{proof}

\begin{claimx}\label{lemma_4.5:claim_E}
$s = n-2$ and $u = m-2$.
\end{claimx}

\begin{proof}[Proof of \textsc{Claim} \ref{lemma_4.5:claim_E}]
We will show that $s = n - 2$, as the equality $u = m-2$ can be proved similarly. 

To begin, we infer from Eq.~\eqref{equ-a^r-b^s(1)} and \textsc{Claim}~\ref{lemma_4.5:claim_C}  
that there exists $j \in \llb 2, n - 1 \rrb$ such that
\begin{equation}\label{equ-ab^j}
ab^j \{1_K, b\}^s \subseteq \{1_K, a\}^{r-1} \{1_K, ab^j\} \{1_K, b\}^s = \{1_K, a\}^r \{1_K, b\}^{s+1}.
\end{equation}
If $ab^p = a^k b^q$ for some $p, q \in \mathbb{Z}$ and $k \in \llb 0, r \rrb$, then  
$a^{k - 1} = b^{p - q}$, which, by Eq.~\eqref{equ-r|a,v|b}(ii), is only possible if $r \mid k - 1$. So, $k = 1$ (as we have $r \ge 2$ by \textsc{Claim}~\ref{lemma_4.5:claim_B}),  
and we conclude from Eq.~\eqref{equ-ab^j} that  
\[
ab^j \{1_K, b\}^s \subseteq a\{1_K, b\}^{s + 1}.
\]
By cancelling $a$ on the left, the last display gives in turn that  
\begin{equation}
\label{equ:something}
b^{j+s} \in b^j \{1_K, b\}^s \subseteq \{1_K, b\}^{s + 1}.
\end{equation}
Now, \textsc{Claim}~\ref{lemma_4.5:claim_C} yields $2 \le j \le n - 1$,  
and thus $s + 1 < j + s$. Consequently, for Eq.~\eqref{equ:something} to hold, it is  
necessary that $j + s \ge n$ (recall that $b$ is a unit of order $n$ in $K$). It follows that  
$j + k = n$ for some $k \in \llb 1, s \rrb$, which, again by Eq.~\eqref{equ:something},  
implies that $
b^{n - 1} = b^{j + k - 1} = b^t$
for a certain $t \in \llb 0, \allowbreak s + \allowbreak 1 \rrb$. However, the latter is only possible if $s + 1 \ge n - 1$, and so,  
by \textsc{Claim}~\ref{lemma_4.5:claim_D}, $s = n - 2$.
\renewcommand{\qedsymbol}{[\textit{Proof of \textsc{Claim} \ref{lemma_4.5:claim_E}}\,]\,$\square$}
\end{proof}

We are ready to finish the proof of Lemma \ref{lem:g(xy)=g(x)g(y)orx^2y^2=1}. In fact, since $y$ is a unit of order $n$ in $H$, we get from \textsc{Claim}~\ref{lemma_4.5:claim_E} that $x^r = y^{n-2} = y^{-2}$. Likewise, 
$y^v = x^{m-2} = x^{-2}$. Therefore, $
x^2 = \allowbreak y^{-v}$ and $y^2 = x^{-r}$. By \textsc{Claim}~\ref{lemma_4.5:claim_B} and  Eq.~\eqref{equ-r|a,v|b}(ii), this shows that $r = v = 2$. As a result, $x^2 = y^{-2}$, that is, 
$x^2 y^2 = 1_H$.
\end{proof}

\begin{lemma}\label{lem-x^2=1}
Let $g$ be the pullback of an isomorphism $f$ from $\mathcal{P}_{\fin,1}(H)$ to $\mathcal{P}_{\fin,1}(K)$, where $H$ and $K$ are cancellative monoids. If $x, y \in H$ are torsion elements with $x^2 = 1_H$ or $y^2 = 1_H$, then $g(xy) = g(x)g(y)$.
\end{lemma}

\begin{proof}
Assume to the contrary that $g(xy) \ne ab$, where $a := g(x)$ and $b := g(y)$. By Lemma~\ref{lem:g(xy)=g(x)g(y)orx^2y^2=1}, we have $x^2 y^2 = 1_H$; and by the hypothesis that $x^2 = 1_H$ or $y^2 = 1_H$, this entails $x^2 = y^2 = 1_H$.  
Consequently,
\begin{equation*}
\{1_H, x\} \{1_H, xy\} \{1_H, y\} 
= \{1_H, x, y, xy, x^2 y, xy^2, x^2 y^2\} 
= \{1_H, x, y, xy\} 
= \{1_H, x\} \{1_H, y\}.
\end{equation*}
By applying $f$ to the last display and using $f(\{1_H, z\}) = \{1_K, g(z)\}$ 
for each $z \in H$, it follows that
\begin{equation}\label{equ:last-display}
ab \ne g(xy) 
\in \{1_K, a\} \{1_K, g(xy)\} \{1_K, b\} 
= \{1_K, a\} \{1_K, b\} 
= \{1_K, a, b, ab\}.
\end{equation}
Now, we gather from
Eq.~\eqref{equ:last-display} and the injectivity of $g$ that $xy \in \{1_H, x, y\}$, which, by the cancellativity of $H$,  
yields either $xy = 1_H$, or $x = 1_H$, or $y = 1_H$. However, if $x = 1_H$ or $y = 1_H$, then $g(xy) = ab$, which is absurd.
So, the only possibility is that $xy = 1_H$, and since $x^2 = 1_H$ (as noted above), we find that
$y = \allowbreak x^2 y = x(xy) = x1_H = x$.
Accordingly, we are guaranteed by Proposition~\ref{prop:g(x^k)=g(x)^k} that
$$
g(xy) = g(x^2) = g(x)^2 = g(x) g(y) = ab,
$$
which is again a contradiction and completes the proof.
\end{proof}

\begin{proposition}\label{pro-main}
Let $g$ be the pullback of an isomorphism $f$ from $\mathcal{P}_{\fin,1}(H)$ to $\mathcal{P}_{\fin,1}(K)$, where $H$ and $K$ are cancellative monoids. If $x, y \in H$ are torsion elements, then $g(xy) = g(x)g(y)$. 
\end{proposition}

\begin{proof}
Suppose for a contradiction that $g(xy) \ne g(x)g(y)$. Then, by Lemmas \ref{lem:g(xy)=g(x)g(y)orx^2y^2=1} and \ref{lem-x^2=1},
\begin{equation}\label{prop_4.7:eq_01}
\text{(i) } x^2y^2 = 1_H
\quad\text{and}\quad
\text{(ii) } x^2 \ne 1_H \ne y^2.
\end{equation}
It follows that $x^4 y^2 = x^2 (x^2 y^2) = x^2 \ne 1_H$. Thus, we infer from Lemma \ref{lem:g(xy)=g(x)g(y)orx^2y^2=1} and Proposition \ref{prop:g(x^k)=g(x)^k} that
\begin{equation}\label{prop_4.7:eq_02}
g(x^2y) = g(x^2)g(y) = g(x)^2 g(y) \ne g(x) g(xy),
\end{equation}
where the inequality at the far right stems from the cancellativity of $K$ and the assumption that $g(xy) \ne g(x) g(y)$.
Then, another call to Lemma \ref{lem:g(xy)=g(x)g(y)orx^2y^2=1}, together with Eqs.~\eqref{prop_4.7:eq_01}(i) and \eqref{prop_4.7:eq_02}, yields
$$
x^2 y^2 = 1_H = x^2(xy)^2 = x^3yxy,
$$
and hence $xyx = y$ (by the cancellativity of $H$). Accordingly, we find that
\[
\{1_H, xy\}\{1_H, x\} = \{1_H, xy, x, xyx\} = \{1_H, x, y, xy\} = \{1_H,x\} \{1_H,y\},
\]
which, by taking the image of the left-most and right-most sides under $f$, leads to
\begin{equation}\label{prop_4.7:eq_03}
g(x) g(y) \ne g(xy) \in \{1_K, g(xy)\} \{1_K, g(x)\} = \{1_K, g(x)\}\{1_K, g(y)\} = \{1_K, g(x), g(y), g(x) g(y)\}.
\end{equation}
Now, $n := \ord_H(x)$ is a positive integer, since $x$ is a torsion element in $H$. Similarly as in the last part of the proof of Lemma \ref{lem-x^2=1}, it is then straightforward from Eqs.~\eqref{prop_4.7:eq_01}(ii) and \eqref{prop_4.7:eq_03} that $xy = 1_H$, and hence $y = x^{-1} = x^{n-1}$. As a result, we conclude from Proposition~\ref{prop:g(x^k)=g(x)^k} that
$$
g(xy) = g(x^n) = g(x)^n = g(x) g(x)^{n-1} = g(x) g(x^{n-1}) = g(x) g(y),
$$ 
which is however a contradiction and completes the proof.
\end{proof}

\section{Main result}

We are now in a position to prove the main result of the paper (Theorem \ref{thm:BG-for-torsion-groups}) and thereby  
show that Question~\ref{ques:BG-like-for-monoids} has a positive answer for the class of torsion groups (Corollary~\ref{cor:4.9}).

\begin{theorem}\label{thm:BG-for-torsion-groups}
Let $H$ and $K$ be cancellative monoids, and suppose that at least one of them is torsion. Assume, in addition, that there exists an isomorphism $f$ from $\mathcal P_{\fin,1}(H)$ to $\mathcal P_{\fin,1}(K)$. Then $H$ and $K$ are both groups, and the pullback of $f$ is an isomorphism from $H$ to $K$.
\end{theorem}

\begin{proof}
Let $g$ be the pullback of the isomorphism $f$. 
By Proposition \ref{prop:fix-ord} and the bijectivity of $g$ (Corollary \ref{cor:iso-induces-a-bijection}), $H$ is torsion if and only if $K$ is. On the other hand, any cancellative torsion monoid is a group. Therefore, $H$ and $K$ are both torsion groups, and it follows from 
Proposition~\ref{pro-main} that $g(xy) = g(x)g(y)$ for all $x, y \in H$. That is, $g$ is an isomorphism from $H$ to $K$. 
\end{proof}

The following corollary is now an obvious consequence of Theorem~\ref{thm:BG-for-torsion-groups}.

\begin{corollary}\label{cor:4.9}
If $H$ and $K$ are groups, at least one of which is torsion, and $\mathcal P_{\fin,1}(H)$ is isomorphic to $\mathcal P_{\fin,1}(K)$, then $H$ is isomorphic to $K$.
\end{corollary}

It remains open whether 
Corollary~\ref{cor:4.9} extends to arbitrary groups, while it is known from \cite{Rago2025} that  
Theorem~\ref{thm:BG-for-torsion-groups} does not carry over to arbitrary cancellative monoids, even in the commutative setting.

\section*{Acknowledgments}

The authors were supported by Grant No.~A2023205045 from the Natural Science Foundation  
of Hebei Province. Weihao Yan was also supported by Austrian FWF Project No.~PAT9756623  
during his visit to the University of Graz in July--August 2025.

The authors are grateful to the anonymous referees of a previous version of the manuscript for their careful proofreading and many
helpful comments that have greatly improved the presentation.

\end{document}